\documentclass[10pt,oneside,a4paper]{article} 		
\usepackage[fleqn]{amsmath} 							
\usepackage{amsthm,amsfonts,latexsym,amssymb,amscd} 	
\usepackage[mathscr]{eucal}   							
\usepackage[all]{xy}										

\usepackage[draft=false]{hyperref} 					

\newcommand{\CC}{{\mathbb{C}}}

\newcommand{\RR}{{\mathbb{R}}}

\newcommand{\As}{{\mathscr{A}}}\newcommand{\Bs}{{\mathscr{B}}}\newcommand{\Cs}{{\mathscr{C}}}
\newcommand{\Ds}{{\mathscr{D}}}\newcommand{\Es}{{\mathscr{E}}}

\newcommand{\Ms}{{\mathscr{M}}}\newcommand{\Ns}{{\mathscr{N}}}
\newcommand{\Ss}{{\mathscr{S}}}

\DeclareFontFamily{U}{rsfs}{\skewchar\font127 }
\DeclareFontShape{U}{rsfs}{m}{n}{%
   <5> <6> rsfs5
   <7> rsfs7
   <8> <9> <10> <10.95> <12> <14.4> <17.28> <20.74> <24.88> rsfs10
}{}
\DeclareSymbolFont{rsfs}{U}{rsfs}{m}{n}
\DeclareSymbolFontAlphabet{\scr}{rsfs}

\newcommand{\Cf}{\scr{C}}\newcommand{\Hf}{\scr{H}}\newcommand{\Kf}{\scr{K}}

\DeclareMathOperator{\id}{Id}
\DeclareMathOperator{\Ob}{Ob}

\DeclareMathOperator{\Hom}{Hom}

\DeclareMathOperator{\Sp}{Sp}

\renewcommand{\emph}{\textbf} 						
\newcommand{\cj}[1]{\overline{#1}}					
\newcommand{\ip}[2]{\langle #1\mid #2\rangle}	
\renewcommand{\iff}{\Leftrightarrow}				

\newcommand{\hlink}[2]{\href{#1}{\texttt{#2}}} 

\newtheorem{theorem}{Theorem}[section]			

\newtheorem{proposition}[theorem]{Proposition}

\newtheorem{definition}[theorem]{Definition}

\numberwithin{equation}{section}  	
\title{\textbf{Kre\u\i n C*-categories\footnote{Published in: Chamchuri Journal of Mathematics 1 (2009) n.2:63-77.}}}

\author{\normalsize Paolo Bertozzini$^a$, Kasemsun Rutamorn\footnote{Current address: Department of Mathematics, Faculty of Education, Dhonburi Rajabhat University, 172 Itsaraphap Road, Thonburi 10600 Bangkok.}$\ ^b$ 
\\
\normalsize  \textit{Department of Mathematics and Statistics, Faculty of Science and Technology}
\\
\normalsize \textit{Thammasat University, Bangkok 12121, Thailand}
\\
\normalsize $^a$e-mail: \texttt{paolo.th@gmail.com}
\\
\normalsize $^b$e-mail: \texttt{kasemamorn270@hotmail.com}
}

\date{\normalsize{published: 14 December 2009, arXiv version: 26 December 2011}}

\begin{document}

\maketitle

\begin{abstract}\noindent 
C*-categories are essentially norm-closed $*$-categories of bounded linear operators between Hilbert spaces. 
The purpose of this work is to identify suitable axioms defining Kre\u\i n C*-categories, i.e.~those categories that play the role of 
C*-categories whenever Hilbert spaces are replaced by more general indefinite inner product Kre\u\i n spaces, and provide some basic examples. Finally we provide a Gel'fand-Na\u\i mark representation theorem for Kre\u\i n C*-algebras and Kre\u\i n 
\hbox{C*-categories}.  

\medskip

\noindent
\emph{Keywords:}
Kre\u\i n space, C*-category. 

\medskip

\noindent
\emph{MSC-2010:}
					47B50, 			
					46C20, 			
					46M15, 			
					46L08. 			
\end{abstract}


\section{Introduction}

It is well know that every C*-algebra is essentially an involutive norm-closed algebra of bounded linear operators on a Hilbert space. 
C*-categories are a generalization of the notion of C*-algebra arising whenever we consider norm-closed families of operators, between (possibly) many Hilbert spaces, that are closed under composition and adjoint.  
They were first introduced in 1985 by P.Ghez-R.Lima-J.Roberts~\cite{GLR} and since then they have been extensively used in algebraic quantum field theory. 
In 2001 P.Mitchener~\cite{M} further examined the definition of C*-categories, studied their $K$-theory and applied them to the 
Baum-Connes conjecture.  

Although indefinite inner product spaces have been considered since the beginning of the last century in phyiscs, Kre\u\i n spaces i.e.~vector spaces equipped with an indefinite inner product, that are complete, were defined (following previous works by L.Pontrjagin~\cite{Po}) by 
J.Ginzburg~\cite{Gi} and E.Scheibe~\cite{Sc} and have been subsequently studied by many mathematicians of the Russian school of M.Kre\u\i n. 

Bounded linear maps on a Kre\u\i n space constitute natural examples of Kre\u\i n C*-algebras 
(they are the analogue of C*-algebras in the case of Kre\u\i n spaces) a concept that was introduced by K.Kawamura~\cite{K} in 2006.

In this work our purpose is to try to provide suitable axioms for the abstract definition of Kre\u\i n C*-categories. In particular our definitions should necessarily include the basic example of bounded linear maps between (possibly) many Kre\u \i n spaces. 
We will also provide a few examples and study some of the properties of Kre\u\i n C*-categories. In particular we will associate a 
Kre\u\i n C*-category to every fundamental symmetry of a Kre\u\i n algebra (or more generally to every fundamental symmetry of a Kre\u\i n 
C*-category) and use well-known Gel'fand-Na\u\i mark representation results for C*-categories in order to provide 
Gel'fand-Na\u\i mark representation theorems for Kre\u\i n C*-algebras and Kre\u\i n C*-categories. 

\bigskip

\textbf{Notes and acknowledgements:}
we thank Dr.Roberto Conti for several discussions and Associate Prof.Wicharn Lewkeeratiyutkul and Associate Prof.Pachara Chaisuriya for kind suggestions.

The following paper (based on K.Rutamorn's Master Degree Thesis in Thammasat University) was originally presented by K.Rutamorn at the ``Annual Pure and Applied Mathematics Conference 2009'', organized on May 25-26 by the Department of Mathematics of Chulalongkorn University, and it appeared in the non-refereed proceedings of the same conference: P.Bertozzini, K.Rutamorn, Kre\u\i n C*-categories, APAM 2009, Collection of Abstract and Presented Papers, 35-42 (2009). 
The work was submitted on 30 July 2009 to Chamchuri Journal of Mathematics and, after minor revision, accepted for publication on 14 December 2009. 
The present file is prepared only for upload to the arXiv and contains a reformatted version of the published paper with minor corrections of misprints and updates in the entries of the original bibliography. 

\section{Preliminaries}

We collect here for the benefit of the reader the main definitions and some properties that will be used in this work. 
For general background in the theory of operator algebras we refer to the books by V.Sunder~\cite{Su}, M.Takesaki~\cite{Ta},  B.Blackadar~\cite{Bl} and for Hilbert \hbox{C*-modules} to E.Lance's book~\cite{La} and the notes by N.Landsman~\cite{L}. 
General definitions and notations from category theory can be obtained from S.MacLane~\cite{Mc} M.Barr-C.Wells~\cite{BW} or R.Geroch~\cite{G} for an elementary introduction.

\subsection{C*-algebras and C*-categories}

\begin{definition}
A complex \emph{C*-algebra} $\As$ is a complex associative algebra (i.e.~a vector space over the complex numbers with a bilinear associative multiplication $\cdot:\As\times\As\to\As$) equipped with a conjugate linear map $*:\As\to\As$ such that 
$(x^*)^*=x$, $(x\cdot y)^*=y^*\cdot x^*$, for all $x,y\in \As$ that is a complete metric space with a norm $\|\cdot\|:\As\to\RR$ such that 
$\|x\cdot y\|\leq\|x\|\cdot\|y\|$ and $\|x^*\cdot x\|=\|x\|^2$, for all $x,y\in \As$. 
The C*-algebra is \emph{unital} if there is an identity element $1_\As$ such that $x\cdot 1_\As=x=1_\As\cdot x$, for all $x\in\As$ and 
$\|1_\As\|=1$.
\end{definition}

\begin{definition}
In a unital C*-algebra $\As$ and element $x\in \As$ is \emph{positive} if it is Hermitian i.e.~$x^*=x$ and its spectrum is positive 
$\Sp(x)\subset \RR_+$, where the spectrum of $x\in \As$ is defined as $\Sp(x):=\{\mu\in \CC \ | \ x-\mu\cdot 1_\As\ \text{is not invertible}\}$. The set of positive elements of $\As$ will be denoted by $\As_>$. 
\end{definition}

Recall that in a unital C*-algebra $x\in \As_>$ if and only if there exists $z\in \As$ such that $x=z^*z$. 

\begin{definition}
A \emph{state} over the unital C*-algebra $\As$ is a linear map $\omega:\As\to\CC$ that is normalized 
i.e.~$\omega(1_\As)=1_\CC$ and positive in the sense that $\omega(x)\in \CC_>=\RR_+$ for all $x\in \As_>$. 

A unital \emph{$*$-homomorphism} $\phi:\As\to\Bs$ between two unital C*-algebras $\As,\Bs$ is a linear map such that 
$\phi(xy)=\phi(x)\phi(y)$ and $\phi(x^*)=\phi(x)^*$, for all $x,y\in \As$ and $\phi(1_\As)=1_\Bs$. 
\end{definition}
States are always Hermitian functionals i.e.~$\omega(x^*)=\cj{\omega(x)}$, for all $x\in\As$. 

\begin{definition} 
A left \emph{Hilbert C*-module} ${}_\As\Ms$ over the C*-algebra $\As$ is a left $\As$-module equipped with an $\As$-valued inner product 
${}_\As\ip{\cdot}{\cdot}:(x,y)\mapsto{}_\As\ip{x}{y}$, i.e.~a sesquilinear map, linear in the left variable, that is Hermitian 
${}_\As\ip{x}{y}^*={}_\As\ip{y}{x}$, positive ${}_\As\ip{x}{x}\in \As_>$, 
non-degenerate ${}_\As\ip{x}{x}=0_\As\iff x=0_\Ms$ in such a way that $\Ms$ becomes a Banach space with the norm 
$\|x\|_\Ms:=\sqrt{{}_\As\ip{x}{x}}$. 
The Hilbert C*-module ${}_\As\Ms$ is \emph{unital} if the C*-algebra $\As$ is unital and $1_\As\cdot x=x$, for all $x\in \Ms$. 

A right Hilbert C*-module $\Ms_\Bs$ over the C*-algebra $\Bs$ is defined in a similar way using a right $\Bs$-module, but in this case the $\Bs$-valued inner product $\ip{x}{y}_\Bs$ is a sesquilinear form that is assumed to be linear in the right entry. 

A \emph{Hilbert C*-bimodule} ${}_\As\Ms_\Bs$ over the C*-algebras $\As$ on the left and $\Bs$ on the right is a left Hilbert C*-module over $\As$ and also a right Hilbert C*-module over $\Bs$ such that $(ax)b=a(xb)$, for all $a\in \As$, $x\in \Ms$, $b\in\Bs$.\footnote{Further conditions might be imposed (such as the adjointability of the right (left) action with respect to the the left (right) inner product, imprimitivity, etc.), but we will not  necessarily impose such requirements.}
\end{definition}

Following P.Ghez-R.Lima-J.Roberts~\cite{GLR} and P.Mitchener~\cite{M} we have the following generalization of the notion of 
C*-algebra. 
\begin{definition}
A complex \emph{C*-category} is category $\Cs$ such that $\Hom_\Cs(A,B)$ is a complex Banach space for all $A,B\in \Ob_\Cs$ that is equipped with a contravariant involutive conjugate linear functor $*:\Cs\to\Cs$ acting identically on the objects such that $\|1_A\|=1$ for all $A\in \Ob_\Cs$, 
$\|x\circ y\|\leq\|x\|\cdot\|y\|$, for all composable $x,y\in \Hom_\Cs$, $\|x^*\circ x\|=\|x\|^2$ for all 
$x\in \Hom_\Cs$ and such that $x^*\circ x$ is a positive element in the C*-algebra $\Hom_\Cs(A,A)$ for all $x\in \Hom_\Cs(B,A)$.
\end{definition}
From the definition we see that in a C*-category $\Cs$, for all objects $A,B\in \Ob_\Cs$, the ``diagonal blocks'' 
$\Hom_\Cs(A,A)$, $\Hom_\Cs(B,B)$ are unital C*-algebras and the the ``off-diagonal blocks'' $\Hom_\Cs(A,B)$ are unital Hilbert C*-bimodules over the C*-algebras $\Hom_\Cs(A,A)$, acting on the left and $\Hom_\Cs(B,B)$ acting on the right. 

Clearly, a C*-category with only one object can be identified with a unital \hbox{C*-algebra} and hence C*-categories are ``many-objects'' versions of unital \hbox{C*-algebras}. 

If $\Hf$ is a family of Hilbert spaces (or more generally Hilbert C*-modules over a given C*-algebra) the collection $\Bs(\Hf)$ of linear continous maps (adjointable maps) between the Hilbert spaces (respectively the modules) in $\Hf$ has the structure of a C*-category with objects given by the spaces in the family $\Hf$, composition, involution and norm given by the usual composition, adjoint and norm of operators.  

\medskip

In the case of C*-categories, the notion of unital \hbox{$*$-homomorphism} is generalized by that of $*$-functor. 
\begin{definition}
Let $\Cs$ and $\Ds$ be two C*-categories. A \emph{covariant $*$-functor} $\phi:\Cs\to\Ds$ is given by a pair of maps $\phi:A\to\phi_A$ between the set of objects and $\phi:x\mapsto \phi(x)$ between the set of morphisms such that $x\in \Hom_\Cs(A,B)$ implies 
$\phi(x)\in \Hom_\Ds(\phi_A,\phi_B)$, $\phi(x\circ y)=\phi(x)\circ\phi(y)$, for all composable $x,y\in \Hom\Cs$ and 
$\phi(x^*)=\phi(x)^*$, for all $x\in \Hom_\Cs$. 
A covariant $*$-functor $\pi:\Cs\to\Bs(\Hf)$ is called a \emph{representation} of the \hbox{C*-category} $\Cs$. 
\end{definition}

We propose the following generalization, in the setting of C*-categories, of the notion of state on a C*-algebra. 
\begin{definition}
A \emph{state} $\omega:\Cs\to\CC$ \emph{on a C*-category} $\Cs$ is map $\omega:\Hom_\Cs\to\CC$ that is linear when restricted to $\Hom_\Cs(A,B)$, for all $A,B\in \Ob_\Cs$, Hermitian i.e.~$\omega(x^*)=\cj{\omega(x)}$, for all $x\in \Hom_\Cs$, 
normalized i.e.~$\omega(1_A)=1_\CC$, for all $A\in \Ob_\Cs$, and positive i.e.~$\omega(x^*\circ x)\geq0$, for all $x\in \Hom_\Cs$. 
\end{definition}

The celebrated Gel'fand-Na\u\i mark-Segal theorem admits a natural ``categorified'' version. 
\begin{theorem} 
Let $\omega$ be a state over the C*-category $\Cs$. There is a representation $\pi_\omega$ of $\Cs$ on a family $H_A$ of Hilbert spaces indexed by objects of $\Cs$ and there exists a family of normalized vectors $\xi_A\in H_A$ such that 
$\omega(x)=\ip{\xi_A}{\pi_\omega(x)\xi_B}_{H_A}$, for all $x\in \Hom_\Cs(A,B)$. 
\end{theorem} 
\begin{proof}
Consider the set $\Ns^\omega:=\{x\in \Hom_\Cs \ | \ \omega(x^*\circ x)=0\}$ and note that $\Ns^\omega$ is a closed left $*$-ideal in the C*-category 
$\Cs$. For all $A,B\in \Ob_\Cs$, the quotient $\Cs_{BC}/\Ns^\omega_{BC}$ is a pre-Hilbert space, with the 
well-defined inner product given by $\ip{y+\Ns^\omega_{BC}}{z+\Ns^\omega_{BC}}=y^*\circ z + \Ns^\omega_{BC}$, for all 
$y,z\in \Cs_{BC}$ and for every $x\in \Cs_{AB}$, we have for all $C\in \Ob_\Cs$ well-defined continuous maps 
$L^C_x:\Cs_{BC}/\Ns^\omega_{BC}\to\Cs_{AC}/\Ns^\omega_{AC}$ given by left multiplication by $x$ 
i.e.~$L^C_x(z+\Ns^\omega_{BC}):=x\circ z + \Ns^\omega_{AC}$. Upon passing to the Hilbert space completion $H_{BC}$ of the 
pre-Hilbert spaces $\Cs_{BC}/\Ns^\omega_{BC}$ the maps $L^C_x$ extend by continuity to continuos linear maps 
$\pi^C(x):H_{BC}\to H_{AC}$. The maps $\pi^C$ satisfy $\pi^C(x\circ y)=\pi^C(x)\circ\pi^C(y)$, $\pi^C(x^*)=\pi^C(x)^*$, 
$\pi^C(1_A)=\id_{H_{AA}}$ and hence we have a representation $\pi_\omega:=\oplus_{C\in \Ob_\Cs}\pi^C$ of $\Cs$ on the families of Hilbert spaces 
$\{H_B:=\oplus_{C\in \Ob_\Cs} H_{BC}\ | \ B\in \Ob_\Cs\}$. 
Defining $\xi_A:= 1_A+\Ns^\omega_{AA}\in H_{AA}\subset \oplus_{C\in\Ob_\Cs} H_{AC}=H_A$, we can check that 
$\omega(x)=\ip{\xi_A}{\pi^C(x)\xi_B}_{H_{AA}}=\ip{\xi_A}{\pi_\omega(x)\xi_B}_{H_A}$. 
\end{proof}
From this result we can obtain a Gel'fand-Na\u\i mark representation theorem for C*-categories (see~\cite{GLR,M}). 
\begin{theorem}\label{th: gn}
A C*-category $\Cs$ admits an isometric representation in the \hbox{C*-category} $\Bs(\Hf)$ of a family of Hilbert spaces $\Hf$. 
\end{theorem}
\begin{proof} 
Let $\Ss_\Cs$ be the set of states defined on $\Cs$. Considering $\pi:=\oplus_{\omega\in \Ss_\Cs}\pi_\omega$ we get a new  representation of $\Cs$. Since $\|\pi(x^*x)\|=\|\pi(x)\|^2$, the representation if isometric if and only if it is isometric on the elements of the diagonal blocks of $\Cs$. 
Now for every element $x\in \Cs_{AA}$ there is at least one state $\omega_x$ such that $\|\pi_{\omega_x}(x)\|=\|x\|$  
and hence $\|\pi(x)\|\geq\|\pi_{\omega_x}(x)|=\|x\|$ and the isometry follows. 
\end{proof}

\begin{definition}
A \emph{C*-envelope} $\Es(\Cs)$ of a C*-category $\Cs$ is given by a unital C*-algebra $\Es(\Cs)$ and  a $*$-functor 
$\iota:\Cs\to\Es(\Cs)$ that satisfy the following universal factorization property: 
\\
for all unital $*$-functors $\phi:\Cs\to \As$ into a unital C*-algebra $\As$, there exists one and only one unital $*$-homomorphism 
$\Phi:\Es(\Cs)\to\As$ such that $\Phi\circ\iota=\phi$. 
\end{definition}
By the universal factorization property above, it is a standard result that the C*-envelope of a C*-category is unique up to unital isomorphims of 
C*-algebras. The existence of a C*-envelope is proved as in~\cite[Page~86]{GLR} taking the inductive C*-limit of the finite matrix 
C*-algebras of the C*-category. 

If the C*-category $\Cs$ has only a finite set of objects $\{A_1,\dots,A_N\}$, the construction of the C*-envelope $\Es(\Cs)$ just deliver the matrix 
C*-algebra of $\Cs$, i.e.~the set of $N\times N$ matrices with coeffients entries in position $(i,j)$ taken from the set 
$\Hom_\Cs(A_j,A_i)$, for $i,j=i,\dots, N$; with line by column multiplication and involution given by $*$-transposition.  

Again, from the universal factorization property of C*-envelopes, we obtain this result. 
\begin{proposition}\label{prop: env}
Any $*$-functor $\phi:\Cs\to\Ds$ between C*-categories induces a unique unital $*$-homomorphism of enveloping C*-algebras 
$\Es(\phi):\Es(\Cs)\to\Es(\Ds)$ such that $\Es(\phi)\circ\iota_\Cs=\iota_\Ds\circ\phi$. 
The map $\Es:\phi\mapsto \Es(\phi)$ is functorial and hence it preserves mono epi and isomorphisms. 
\end{proposition} 

\subsection{Kre\u\i n Spaces and Kre\u\i n C*-algebras}

Referring to M.Dritschel-J.Rovnyak~\cite{DR} for a more complete treatment, we provide some basic definition and result on Kre\u\i n spaces. 
\begin{definition}
A complex \emph{Kre\u\i n space} is a complex vector space $K$, equipped with a Hermitian sesquilinear form (linear in the second variable) 
$(x,y)\mapsto \ip{x}{y}$, for all $x,y\in K$, that admits at least one direct sum decomposition $K=K_+\oplus K_-$ in orthogonal subspaces such that $K_+$ is a Hilbert space with the restriction of the sesquilinear form and $K_-$ is a Hilbert space with the restriction of the  opposite of sesquilinear form on $K$. 
Any such direct sum decomposition is called a \emph{fundamental decomposition} of the Kre\u\i n space and the linear operator 
\hbox{$J:K\to K$} defined by $J:x_++x_-\mapsto x_+- x_-$ is called a \emph{fundamental symmetry} of the Kre\u\i n space. 
\end{definition}
Of course there is a bijective correspondence between fundamental symmetries and fundamental decompositions of a Kre\u\i n space. A Kre\u\i n space usually admits several different fundamental symmetries. To every fundamental symmetry \hbox{$J:K\to K$} there is a unique associated Hilbert space $|K|_J:=K_+\oplus (-K_-)$ where $-K_-$ denotes the Hilbert space obtained by considering on $K_-$ the opposite of the inner product defined on $K$. 

\begin{proposition}
All the norms $\|\cdot\|_{|K|_J}$ on $K$ obtained from the Hilbert spaces associated to the fundamental symetries $J$ are equivalent and induce a unique topology on the Kre\u\i n space $K$ called the \emph{strong} topology.  
\end{proposition}

The previous definitions and results have been extended to Kre\u\i n C*-modules (see the details in S.Kaewumpai's thesis~\cite{Ka}).
\begin{definition}
A (unital) right \emph{Kre\u\i n C*-module} $\Ms_\As$ over the (unital) \hbox{C*-algebra} $\As$ is a (unital) module over $\As$ equipped with a Hermitian sesquilinear form $(x,y)\mapsto \ip{x}{y}_\As$, for all $x,y\in \Ms$, that admits at least one direct sum decomposition 
$\Ms=\Ms_+\oplus \Ms_-$ in orthogonal $\As$-submodules such that $\Ms_+$ and $-\Ms_-$ are Hilbert C*-modules over 
$\As$ (again here $-\Ms_-$ denotes the module $\Ms_-$ with the restriction of the opposite of the sequilinear form on $\Ms$). 
\end{definition}
Fundamental symmetries are defined in the same way in this more general context of Kre\u\i n C*-modules.
\begin{proposition}
On a Kre\u\i n C*-module $\Ms_\As$ over the C*-algebra $\As$ there is a unique strong topology that is induced by any of the equivalent norms on the Hilbert C*-modules $|\Ms_\As|_J:=\Ms_+\oplus (-\Ms_-)$ given, for $x\in \Ms_\As$, by
\begin{equation*} 
\|x\|_{|\Ms|_J}:=\sqrt{\|\ip{x_+}{x_+}_\As-\ip{x_-}{x_-}_\As\|}.
\end{equation*} 
\end{proposition}

It is well-known, from the Gel'fand-Na\u\i mark representation theorem, that the axioms listed in the definition of C*-algebra characterise those 
$*$-algebras that are isomorphic to closed subalgebras of continuous linear operators on a Hilbert space. 
Similarly the axioms in the definition of C*-category characterise those involutive categories that are isomorphic to closed  
$*$-subcategories of operators between Hilbert spaces. 

A natural problem arising in the setting of Kre\u\i n spaces is the study of the axiomatic characterization of those $*$-algebras and 
$*$-categories that are isomorphic to ``suitable'' closed $*$-algebras, or closed $*$-categories, of linear continuous operators between Kre\u\i n spaces. 

The following is a variation of K.Kawamura~\cite{K} definition of Kre\u\i n C*-algebra.\footnote{Attention that in K.Kawamura's original definition a 
Kre\u\i n C*-algebra is a Banach algebra, but here we only assume that a Kre\u\i n C*-algebra is ``Banachable'' and so its norm is not necessarily fixed.}  

\begin{definition}\label{def: krein}
A complex \emph{Kre\u\i n C*-algebra} is a complex $*$-algebra that admits at least one Banach norm and at least one 
$*$-automorphism $\alpha:\As\to\As$ such that $\alpha\circ\alpha=\id_\As$ and $\|\alpha(x^*)x\|=\|x\|^2$, for all $x\in \As$. 
\end{definition}
An automorphism that satisfies the previous properties is called a \emph{fundamental symmetry} of the Kre\u\i n C*-algebra. 
Every Kre\u\i n C*-algebra usually admits several different fundamental symmetries. For every fundamental symmetry $\alpha$ the 
Kre\u\i n C*-algebra $\As$ becomes a C*-algebra with the new involution defined by $x^{\dag_\alpha}:=\alpha(x^*)$, for $x\in \As$, and so (since in a C*-algebra the norm is unique) for every fundamental symmetry there exists only one norm, here denoted by 
$\|\cdot\|_\alpha$, that satisfies the C*-property in the previous definition. 

Examples of Kre\u\i n C*-algebras are immediately found in algebras of operators on Kre\u\i n spaces and more generally Kre\u\i n 
C*-modules~\cite{Ka}.
\begin{theorem} 
Let $\Ms_\As$ be a unital Kre\u\i n C*-module over the unital C*-algebra $\As$. The $*$-algebra $\Bs(\Ms_\As)$ of adjointable operators on $\Ms_\As$ is a Kre\u\i n C*-algebra. 
\end{theorem}
In particular for every Kre\u\i n space $K$, the algebra of linear continuous maps $\Bs(K)$ is a Kre\u\i n C*-algebra and similarly all the  $*$-subalgebras of $\Bs(K)$ that are invariant for at least one of the fundamental symmetries of $\Bs(K)$ and that are 
closed in the strong topology induced by the strong topology of $K$ are Kre\u\i n C*-algebras.\footnote{As suggested by the referee, the fact that a 
$*$-subalgebra $\As$ of $\Bs(K)$ that is closed in the strong topology is not necessarily a Kre\u\i n C*-algebra (unless one of the fundamental symmetries of $\Bs(K)$ restricts to $\As$), might suggest further improvements or generalizations in the definition of Kre\u\i n C*-algebra. In this work we will limit ourselves to the given definition~\ref{def: krein} and we will assume that  the fundamental symmetries of a Kre\u\i n C*-subalgebra of 
$\Bs(K)$ are restrictions of fundamental symmetries of $\Bs(K)$ and hence every Kre\u\i n C*-subalgebra of $\Bs(K)$ is necessarily stable by the action of one of the fundamental symmetries of $\Bs(K)$.}

In order to probe if the definition of Kre\u\i n C*-algebra is suitable for an axiomatic characterization of those closed $*$-algebras of adjointable operators on Kre\u\i n spaces or Kre\u\i n C*-modules that are stable for the adjoint action of one of the fundamental symmetries of the space, we will make use here of techniques of representation of 
C*-categories and we will begin to develope here a theory of Kre\u\i n C*-categories.  

\section{Kre\u\i n C*-categories}

As immediate categorical generalization of the notion of Kre\u\i n C*-algebra, we propose the following definition.  
\begin{definition}
A \textbf{Kre\u\i n C*-category} is a $*$-category $\Cs$ that admits at least one norm making all the normed spaces $\Hom_\Cs(A,B)$ Banach spaces, for all \hbox{$A,B\in \Ob_\Cs$} and at least one covariant $*$-functor $\alpha:\Cs \to \Cs$ such that:  
\begin{gather}
\alpha\circ\alpha=\id_\Cs,  
\\
\alpha(A)=A\ , \quad \forall A\in \Ob_\Cs, 
\\
\|\alpha(x^*)x\|=\|x\|^2, \quad \forall x\in \Hom_\Cs(A,B), \ \forall A,B\in \Ob_\Cs, 
\\
\forall x\in \Hom_\Cs(B,A), \ \forall A,B\in \Ob_\Cs \ 
\alpha (x^*)x \in \Hom_\Cs(A,A)_>.
\end{gather}
Note that from the first three axioms above, for all objects $A$, the set $\Hom_\Cs(A,A)$ is always a Kre\u\i n C*-algebra and so it is a C*-algebra when the involution is defined as $x^{\dag_\alpha}:=\alpha(x^*)$ so that the last axiom above tells that $\alpha(x^*)x$ is always a positive element in the C*-algebra $\Hom_\Cs(A,A)$ with involution $\dag_\alpha$.   
\end{definition} 
An alternative way to state the last axiom would be to say that for all $x\in \Hom_\Cs(B,A)$, there exists 
$z\in\Hom_\Cs(A,A)$ such that $\alpha(x^*)x=\alpha(z^*)z$. One might also say that the spectrum of $\alpha(x^*)x$ is always positive. 

\medskip

We provide immediately some elementary examples of Kre\u\i n C*-categories. 

\begin{proposition}\label{prop: kc*}
Let $\Cs$ be a C*-category and $\alpha:\Cs\to\Cs$ be a covariant $\dag$-functor such that 
$\alpha\circ\alpha=\id_\Cs$, and $\alpha(A)=A$, for all $A\in \Ob_\Cs$. 
Then $\Cs$ becomes a Kre\u\i n C*-category with the new involution $x^{*_\alpha}=\alpha(x^\dag)$.
\end{proposition}

\begin{proposition}
Let $\Kf_\CC$ be the class of Kre\u\i n spaces with composition of linear continuous mapping and involution given by the Kre\u\i n space adjoint. 
Then $\Kf_\CC$ is a Kre\u\i n C*-category.

More generally let $\Kf_\As$ be the class of Kre\u\i n C*-modules over the C*-algebra $\As$ with composition of Kre\u\i n adjointable maps and their 
Kre\u\i n adjoint as involution. Then $\Kf_\As$ is a Kre\u\i n C*-category.
\end{proposition}

We proceed now to find examples of Kre\u\i n C*-categories naturally associated to Kre\u\i n algebras equipped with a given  fundamental symmetry. 

Motivated somehow from the definition of $*$-bimodule given by N.Waever~\cite{We2}, \cite[Definition~9.6.1]{We1}, we propose the following definition of involutive C*-bimodule:
\begin{definition}
An \textbf{involutive Hilbert C*-bimodule over the C*-algebra $\mathscr A$} is a Hilbert C*-bimodule $\Ms$ over $\As$ that is further equipped with a map \hbox{$*:\Ms\to \Ms$} such that:
\begin{gather*}
(x^*)^*=x, \quad \forall x\in \Ms, 
\\
(ax+y)^*=x^*a^*+y^*, \quad \forall x,y\in \Ms, \ \forall a\in \As, 
\\
{}_{\As}\ip{x}{y}^*=\ip{y^*}{x^*}_{\As}, \quad \forall x,y\in \Ms.
\end{gather*}
\end{definition}

The module involution above essentially provides an indentification of the Rieffel dual of the Hilbert C*-bimodule $\Ms$ with $\Ms$ itself and so it is not surprising the following reformulation of the linking C*-category of a Hilbert C*-bimodule. 

\begin{proposition}
Let $\As$ be a unital C*-algebra and let $\Ms$ be a unital involutive Hilbert C*-bimodule over $\As$ such that 
\begin{equation*}
{}_{\As}\ip{x}{y} z=x\ip{y}{z}_{\As}, \quad \forall x,y,z\in \Ms. 
\end{equation*}
There exists a C*-category $\Cs$ with two objects $\Ob_{\Cs}:=\{+,-\}$ with morphisms 
\begin{gather*}
{\Cs}_{++}:=\As=:{\Cs}_{--}, \qquad 
{\Cs}_{+-}:=\Ms=:{\Cs}_{-+}, 
\end{gather*}
and compositions given by:
\begin{gather*}
(a,b)\mapsto a\cdot_\As b\in {\Cs}_{++}, \quad \forall a,b\in {\Cs}_{++}, 
\\
(a,b)\mapsto a\cdot_\As b\in {\Cs}_{--}, \quad \forall a,b\in {\Cs}_{--}, 
\\
(a,x)\mapsto a\cdot_\Ms x\in {\Cs}_{+-}, \quad \forall a\in {\Cs}_{++}, \ \forall x\in {\Cs}_{+-} 
\\
(y,b)\mapsto y\cdot_\Ms b\in {\Cs}_{+-}, \quad \forall b\in {\Cs}_{--}, \ \forall y\in {\Cs}_{+-} 
\\
(x,a)\mapsto x\cdot_\Ms a\in {\Cs}_{+-}, \quad \forall a\in {\Cs}_{++}, \ \forall x\in {\Cs}_{-+} 
\\
(b,y)\mapsto b\cdot_\Ms y\in {\Cs}_{+-}, \quad \forall b\in {\Cs}_{--}, \ \forall y\in {\Cs}_{-+} 
\\
(x,y)\mapsto \ip{x^*}{y}_{\As}\in {\Cs}_{++}, \quad \forall x\in {\Cs}_{+-}, \ \forall y \in {\Cs}_{-+}, 
\\
(y,x)\mapsto {}_{\As}\ip{x}{y^*}\in {\Cs}_{--}, \quad \forall x\in {\Cs}_{+-}, \ \forall y \in {\Cs}_{-+}.
\end{gather*}
\end{proposition}
We will also denote by $[\As,\Ms]$ the previous C*-category.  

\begin{theorem}
To every unital Kre\u\i n C*-algebra $\As$ with a given fundamental symmetry $\alpha:\As\to\As$, we can associate a Kre\u\i n C*-category 
$\Cs_{(\As,\alpha)}:=[{\As}_+,{\As}_-]$ with two objects. 
\end{theorem}
\begin{proof} 
For a given fundamental symmetry $\alpha$, the Kre\u\i n C*-algebra $\As$ becomes a \hbox{C*-al}gebra, here denoted by $\As^\alpha$, with involution $x^{\dag_\alpha}:=\alpha(x^*)$ with respect to a unique C*-norm 
$\|\cdot\|_\alpha$ and $\alpha:\As\to\As$ becomes a unital $\dag_\alpha$-automorphism of this C*-algebra $\As^\alpha$. 
The set $\As_+^\alpha$ of even elements of $\As^\alpha$ under the action of $\alpha$ is a unital C*-algebra and the set 
$\As_-^\alpha$ of odd elements is a unital involutive Hilbert C*-bimodule over $\As_+^\alpha$ so that, by the previous proposition we can construct a 
C*-category $\Cs_{(\As,\alpha)}^\alpha:=[\As_+^\alpha,\As_-^\alpha]$. 
Note that $\alpha$ induces by restriction a $\dag_\alpha$-functor of the C*-category 
$\Cs_{(\As,\alpha)}^\alpha$ that acts identically on the objects and it is involutive, 
i.e.~$\alpha\circ\alpha=\id$, and so by proposition~\ref{prop: kc*} we have that 
$\Cs_{(\As,\alpha)}^\alpha$ becomes a Kre\u\i n \hbox{C*-category}, here simply denoted by $[\As_+,\As_-]$, with respect to the involution given by $x\mapsto x^*=\alpha(x^{\dag_\alpha})$.  
\end{proof} 

\begin{proposition}\label{prop: isoenv}
For a Kre\u\i n C*-algebra $\As$ with a given fundamental symmetry $\alpha$, there is a canonical unital $*$-isomorphism 
$\Psi:\As^\alpha\to\Es(\Cs_{(\As,\alpha)})$. 
\end{proposition}
\begin{proof}
The inclusion of $\As_+$ and $\As_-$ into $\As^\alpha:=\As_+\oplus \As_-$ induces a $*$-functor $\iota$ from 
$\Cs_{(\As,\alpha)}^\alpha$ to 
$\As^\alpha$ and for every $*$-functor $\phi:\Cs^\alpha_{(\As,\alpha)}\to\Bs$ to a unital C*-algebra $\Bs$, defining 
$\Phi:\As^\alpha\to\Bs$ by 
$\Phi(x_++x_-):=\phi(x_+)+\phi(x_-)$ we check that $\Phi$ is the unique unital $*$-homomorphism such that 
$\Phi\circ\iota=\phi$ and hence $\As^\alpha$ is a \hbox{C*-envelope} of $\Cs^\alpha_{(\As,\alpha)}$ and by the unique factorization property of the 
C*-envelope we get that $\As^\alpha$ is canonically isomorphic to $\Es(\Cs^\alpha_{(\As,\alpha)})$. 
\end{proof}

As an immediate application we obtain the following generalization of Gel'fand-Na\u\i mark representation theorem for unital Kre\u\i n 
C*-algebras. 
\begin{theorem} 
Every unital Kre\u\i n C*-algebra admits at least one faithful representation on a Kre\u\i n space and it is isomorphic to a closed unital 
$*$-subalgebra of a Kre\u\i n C*-algebra of continuous operators on a Kre\u\i n space. 
\end{theorem} 
\begin{proof}
Let $\As$ be a unital Kre\u\i n C*-algebra, let $\alpha$ be one of its fundamental symmetries and $\|\cdot\|_\alpha$ the associated 
C*-norm. Consider now the C*-category $\Cs^\alpha_{(\As,\alpha)}:=[\As^\alpha_+,\As^\alpha_-]$ constructed above and note that, by 
Gel'fand-Na\u\i mark representation theorem for C*-categories~\ref{th: gn} there is an isometric $\dag_\alpha$-isomorphic functor $\rho$ of 
$\Cs^\alpha_{(\As,\alpha)}$ onto a closed C*-category $\rho(\Cs^\alpha_{(\As,\alpha)})$ of linear continuous operators between two Hilbert spaces $H_+$ and $H_-$. 
By proposition~\ref{prop: env} the \hbox{$*$-functor} $\rho$ can be lifted to a unital $\dag_\alpha$-isomorphism 
$\Es(\rho):\Es(\Cs^\alpha_{(\As,\alpha)})\to \Es(\rho(\Cs^\alpha_{(\As,\alpha)}))$ of the C*-envelopes. 

Since $\rho(\Cs^\alpha_{(\As,\alpha)})$ is a C*-category of linear continuous operators between two Hilbert spaces $H_+,H_-$, its 
C*-envelope coincides with the $2\times 2$ matrix \hbox{C*-algebra} with entries from $\rho(\Cs^\alpha_{(\As,\alpha)})$ acting on the Hilbert space $H_+\oplus H_-$ that is a closed $*$-subalgebra of $\Bs(H_+\oplus H_-)$.
Furthermore, by proposition~\ref{prop: isoenv} we can take $\Es(\Cs^\alpha_{(\As,\alpha)})=\As^\alpha$.

Consider now the Kre\u\i n space $K:=H_+\oplus (-H_-)$ obtained reversing the sign of the inner product on $H_-$ in the othogonal direct sum 
$H_+\oplus H_-$ and the fundamental symmetry $J: K\to K$ associated to such fundamental decomposition 
$J:\xi_++\xi_-\mapsto \xi_+-\xi_-$. 

The map $\alpha_J:\Bs(K)\to \Bs(K)$ given by $\alpha_J(T):=JTJ$ is a fundamental symmetry of the Kre\u\i n C*-algebra $\Bs(K)$ and restricted to the Kre\u\i n C*-category $\rho(\Cs_{(\As,\alpha)})$ gives a $*$-functor such that $\rho(\alpha(x))=\alpha_J(\rho(x))$, for all $x\in \Cs_{(\As,\alpha)}$ and that, by proposition~\ref{prop: env}, lifts to a unital $*$-automorphism of 
$\Es(\rho(\Cs_{(\As,\alpha)}))\subset\Bs(K)$ with $\Es(\rho)\circ\alpha=\alpha_J\circ\Es(\rho)$. 

The $*$-isomorphism of C*-algebras $\Es(\rho):\As^\alpha\to \Es(\rho(\Cs^\alpha_{(\As,\alpha)}))$ becomes a unital 
$*$-isomorphism of Kre\u\i n C*-algebras $\Es(\rho):\As\to\Es(\rho(\Cs_{(\As,\alpha)}))^{\alpha_J}$, where the involution in 
$\Es(\rho(\Cs_{(\As,\alpha)}))^{\alpha_J}$ is given by $T^*:=\alpha_J(T^\dag)$, with $T^\dag$ denoting the Hilbert space adjoint in 
$\Bs(H_+\oplus H_-)$. 
\end{proof}

More generally we can obtain a Gel'fand-Na\u\i mark representation for Kre\u\i n \hbox{C*-categories}.
There are other possible ways to prove the result: for example using the Gel'fand-Na\u\i mark theorem for Kre\u\i n C*-algebras applied to the envelope of the Kre\u\i n C*-category, but here we provide only a sketch of a direct proof that generalizes the argument provided in the previous theorem. 
\begin{theorem}
Let $\alpha$ be a fundamental symmetry of a Kre\u\i n C*-category $\Cs$. 
There is a faithful representation $\pi:\Cs\to\Bs(\Kf)$ that is $\alpha$-covariant i.e.~there is an isomorphism of the Kre\u\i n 
C*-category with a closed Kre\u\i n C*-category in $\Bs(\Kf)$ and there exists a family of fundamental symmetries $J_K$, with 
$K\in \Kf$, such that $\pi\circ\alpha=\alpha_{J}\circ\pi$. 
\end{theorem}
\begin{proof}
To every Kre\u\i n C*-category $\Cs$ with fundamental symmetry $\alpha$ and objects $\Ob_\Cs$, we associate another Kre\u\i n 
C*-category $\Cf_{(\Cs,\alpha)}$ with a ``doubled family of objects'' 
\begin{equation*}
\Ob_{(\Cf_{(\Cs,\alpha)})}:=\Ob_\Cs\times\{\pm\}.
\end{equation*} 
The families of morphisms are given by 
\begin{gather*}
\Hom_{\Cf_{(\Cs,\alpha)}}(A\pm,B\pm):=\{x\in \Hom_\Cs(A,B) \ | \ \alpha(x)=x\}, \quad \forall A,B\in \Ob_\Cs,  
\\
\Hom_{\Cf_(\Cs,\alpha)}(A\pm,B\mp):=\{x\in \Hom_\Cs(A,B) \ | \ \alpha(x)=-x\}, \quad \forall A,B\in \Ob_\Cs.
\end{gather*}
Consider the associated C*-category $\Cf^\alpha_{(\Cs,\alpha)}$ with involution $x\mapsto \alpha(x^*)$ and by Gel'fand-Na\u\i mark theorem for 
C*-categories, we obtain a faithful representation $\rho$ in a family of Hilbert spaces indexed by the family of objects 
$\Ob_\Cs\times\{\pm\}$. Define a family of fundamental symmetries $J_A: H_{A+}\oplus H_{A-}\to H_{A+}\oplus H_{A_-}$, indexed by $A\in \Ob_\Cs$, reverting the sign of the $H_{A-}$ component. 
Note that $\rho\circ \alpha=\alpha_J\circ\rho$. 
We can define a notion of \emph{partial envelope} of a ``doubly indexed'' C*-category with respect to one family of indexes (in our case $\{\pm\}$) via the usual universal factorization properties and note that we can assume 
$\Es_\pm(\Cf^\alpha_{(\Cs,\alpha)})=\Cs^\alpha$. 
Furthermore $\Es_\pm(\rho(\Cf^\alpha_{(\Cs,\alpha)}))$ consists of the C*-category with objects 
$\{H_{A_+}\oplus H_{A-}\ | \ A\in \Ob_\Cs\}$ and morphisms $T:H_{A_+}\oplus H_{A-}\to H_{B_+}\oplus H_{B-}$of the form 
\begin{equation*}
T:=
\begin{bmatrix}
\rho(x_+) & \rho(x_-)
\\
\rho(x_-) & \rho (x_+)
\end{bmatrix}, \quad  x\in \Hom_\Cs(A,B), 
\end{equation*}
Since $\Es_{\pm}(\rho):\Cs^\alpha\to\Es_{\pm}(\Cf^\alpha_{(\Cs,\alpha)})$ is an isomorphism of the C*-category $\Cs^\alpha$ onto a closed 
sub-C*-category of $\Bs(\Hf)$ such that $\Es_{\pm}(\rho)\circ\alpha=\alpha_J\circ\Es_{\pm}(\rho)$, defining the family of 
Kre\u\i n spaces $K_A:=H_{A+}\oplus (-H_{A-})$ indexed by $\Ob_\Cs$, we have that $\Es_{\pm}(\rho)$ is an isomorphism of the 
Kre\u\i n C*-category $\Cs$ with the Kre\u\i n C*-category 
$\Es_{\pm}(\Cf^\alpha_{(\Cs,\alpha)})^{\alpha_J}=\Es_{\pm}(\Cf_{(\Cs,\alpha)})$ that is a closed $*$-subcategory of $\Bs(\Kf)$. 
\end{proof}

\section{Final Remarks}

We have proposed here a tentative definition of Kre\u\i n C*-category that, following the general ideas outlined 
in~\cite[Section~4.2]{BCL1}, should provide a ``horizontal categorification'' of the notion of Kre\u\i n C*-algebra given by K.Kawamura. 
Using methods from the theory of C*-categories, we have also described a Gel'fand-Na\u\i mark representation theorem for Kre\u\i n C*-algebras and more generally Kre\u\i n C*-categories that should somehow justify the abstract axioms for Kre\u\i n C*-categories. 
It is of course possible to try to develop in all the details a theory of Kre\u\i n C*-categories (and also of Kre\u\i n W*-categories) that follows in parallel the theory of C*-algebras, here we only started in these directions. 

As described in~\cite{BCL1,BCL2}, C*-categories are just a special case of a more general notion of Fell C*-bundle. In a similar way, it is possible to give a definition of ``Kre\u\i n Fell bundles'' and ``Kre\u\i n spaceoids'' that might later be useful in the study of general spectral theories for (possibly non-commutative) Kre\u\i n C*-algebras (or categories). 
We leave these interesting points to future work.

\end{document}